\documentclass[12pt, paper=a4, fleqn, pdfencoding=unicode]{scrartcl}

\usepackage[utf8]{inputenc}

\usepackage{amsmath}
\usepackage{amssymb}
\usepackage{mathtools}

\mathtoolsset{centercolon}

\DeclarePairedDelimiter\paren{\lparen}{\rparen}
\DeclarePairedDelimiter\brackets{[}{]}
\DeclarePairedDelimiter\braces{\{}{\}}
\DeclarePairedDelimiter\angles{\langle}{\rangle}
\DeclarePairedDelimiter\ceil{\lceil}{\rceil}
\DeclarePairedDelimiter\floor{\lfloor}{\rfloor}
\DeclarePairedDelimiter\abs{\lvert}{\rvert}

\DeclarePairedDelimiterX{\closedStochasticInterval}[1]{[}{]}{\!\delimsize[#1\delimsize]\!}
\DeclarePairedDelimiterX{\leftOpenStochasticInterval}[1]{]}{]}{\!\delimsize]#1\delimsize]\!}
\DeclarePairedDelimiterX{\rightOpenStochasticInterval}[1]{[}{[}{\!\delimsize[#1\delimsize[\!}

\newcommand{\aldousStoppingTime}{{\hat\tau}}

\newcommand{\weakConvergenceTo}{\Rightarrow}

\newcommand{\diff}{\,\mathrm d}

\newcommand{\hittingTimeTo}[1]{H^{#1}}
\newcommand{\hittingTimeFromTo}[2]{H^{#1 \,\to\, #2}}

\newcommand{\indicator}{\mathds{1}}
\newcommand{\EE}{\mathds{E}}

\newcommand{\NN}{\mathds{N}}
\newcommand{\PP}{\mathds{P}}

\newcommand{\RR}{\mathds{R}}

\newcommand{\scE}{\mathcal{E}}
\newcommand{\scF}{\mathcal{F}}
\newcommand{\scG}{\mathcal{G}}

\usepackage{amsthm}
\usepackage[english]{babel}
\usepackage{calc}
\usepackage{caption}
\usepackage{subcaption}
\usepackage{color}
\usepackage{csquotes}
\usepackage{dsfont}
\usepackage{enumitem}
\usepackage{graphicx}
\usepackage[pdfborder={0 0 0}]{hyperref}
\usepackage{import}
\usepackage{microtype}
\usepackage{nameref}
\usepackage[section]{placeins}
\usepackage[noabbrev,capitalize]{cleveref}

	\newtheoremstyle{boldremark}
		{\topsep}   
		{\topsep}   
		{}          
		{}          
		{\bfseries} 
		{.}         
		{.5em}      
		{}          

	\newtheorem{theorem}{Theorem}[section]
	\newtheorem{proposition}[theorem]{Proposition} 
	\newtheorem{lemma}[theorem]{Lemma} 
	\newtheorem{corollary}[theorem]{Corollary}

	\theoremstyle{boldremark}
	\newtheorem{remark}[theorem]{Remark}

\AtBeginDocument{%
	\crefname{equation}{equation}{equations}%
}

\numberwithin{equation}{section}
\numberwithin{theorem}{section}

\author{%
Dirk Becherer,
Todor Bilarev\footnote{Support by German Science foundation DFG via Berlin Mathematical School BMS and 
research training group RTG1845 StoA  is gratefully acknowledged.} ,
Peter Frentrup\footnote{Email addresses: becherer,bilarev,frentrup@math.hu-berlin.de}
\\
Institute of Mathematics, Humboldt-Universität zu Berlin
}

\title{Approximating diffusion reflections at elastic boundaries}

\begin{document}

\maketitle

\begin{abstract}
We show  a probabilistic functional limit result
for
one-dimensional diffusion processes that are reflected at an elastic boundary 
which is
a 
function of the reflection local time. 
Such processes are constructed as  limits of  a sequence of diffusions which are discretely reflected by small jumps at an elastic boundary, with reflection local times being approximated by $\varepsilon$-step processes.
The construction yields the Laplace transform of the inverse local time for reflection.
Processes and approximations of this type play a role in finite fuel problems of singular stochastic control. 
	
	\vspace*{2ex}
	\textbf{Keywords}: Reflected diffusion, elastic boundary, inverse local time, Laplace transform
	
	\textbf{MSC2010 subject classifications}:  60F17, 60J50, 60J55, 60J60, 65C30
\end{abstract}

\section{Introduction}

The classical Skorokhod problem is that of reflecting a path at a boundary. It is a standard tool to construct solutions to SDEs with reflecting boundary conditions.
The fundamental example is Brownian motion with values in $[0,\infty)$ being reflected at a constant boundary at zero, solved by Skorokhod~\cite{Skorokhod61}.
Starting with Tanaka~\cite{Tanaka79}, well-known generalizations concern diffusions in multiple dimensions with normal or oblique reflection at the boundary of some given (time-invariant) domain in the Euclidean space of certain smoothness or other kinds of regularity, cf.\ e.g.\ \cite{LionsSznitman84,DupuisIshii93}.
Other generalizations admit for an a-priori given but time-dependent boundary, see for instance \cite{NystromOnskog10}.

Our contribution is a functional limit result for reflection at a boundary which is a function of the reflection local-time $L$, for  general one-dimensional  diffusions $X$. 
Because of the mutual interaction between boundary and diffusion,  
see~\cref{fig:elastic boundary in X-t-space}, we call the boundary \emph{elastic}.
\begin{figure}[ht]
	\centering
	\begin{subfigure}{0.56\textwidth}
		\centering
		\includegraphics[height=4.2cm]{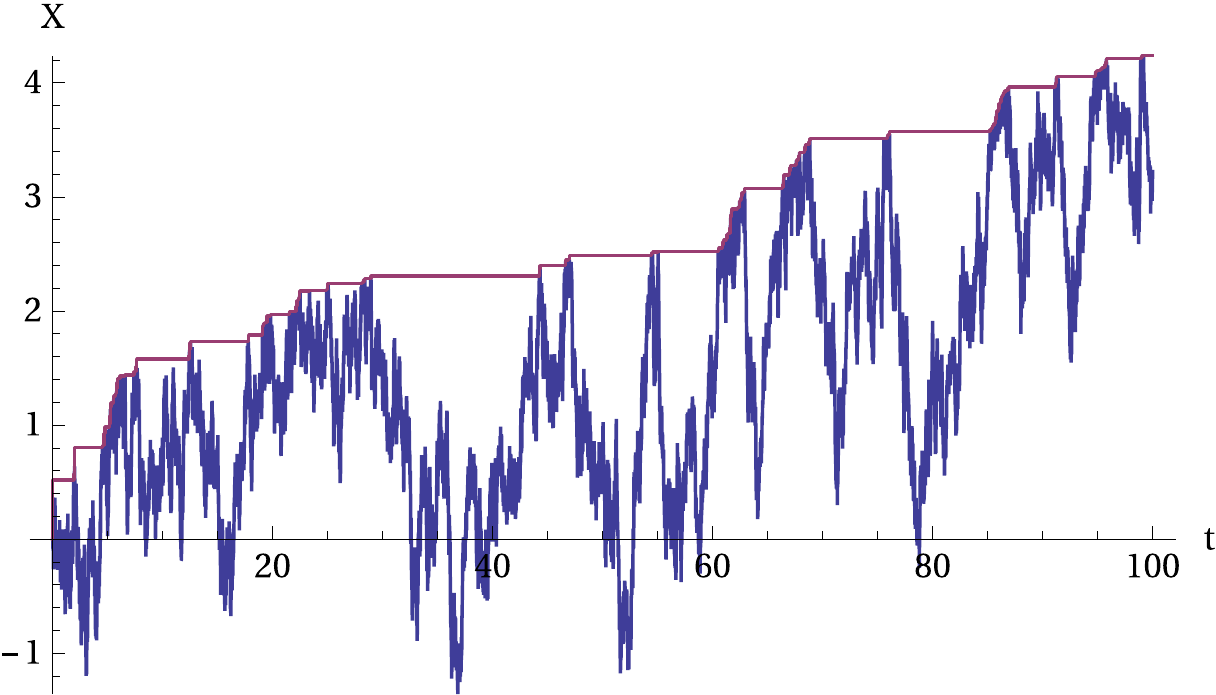}
		\caption{$X$ against real time $t$.}
		\label{fig:elastic boundary in X-t-space}
	\end{subfigure}
	\begin{subfigure}{0.42\textwidth}
		\centering
		\includegraphics[height=4.2cm]{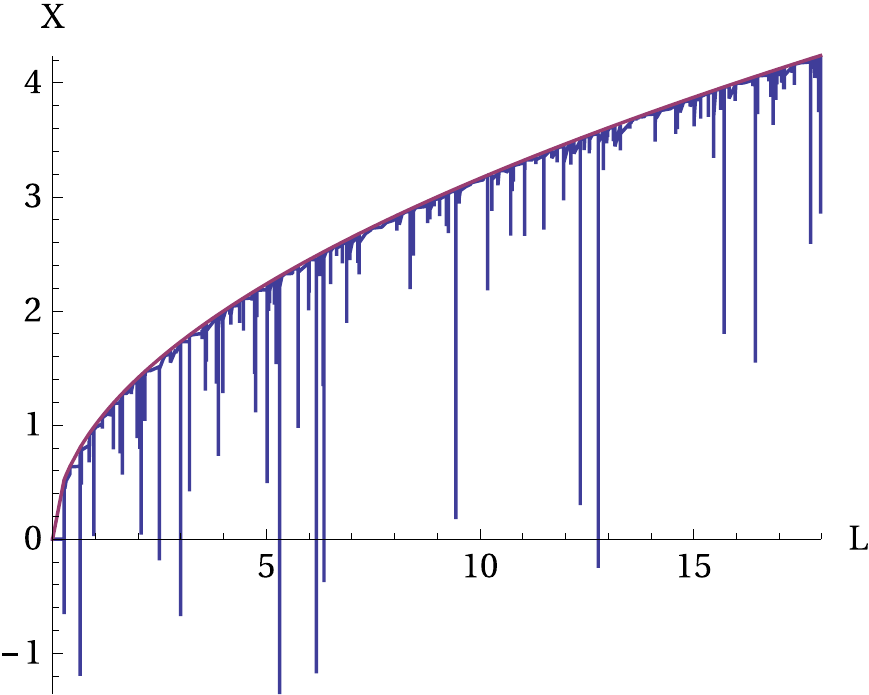}
		\caption{$X$ against local time $L$.}
		\label{fig:fixed boundary in X-L-space}
	\end{subfigure}
	\caption{Example. Brownian motion $X_t$ (blue) reflected at the elastic boundary $g(L)=\sqrt{L}$ (purple), 
where $L$ is the reflection local time of $X$ at boundary $g(L)$.}
\end{figure}
Such elastic boundaries appear typically in solutions to singular control problems of finite fuel type, where the optimal control is the reflection local time that keeps a diffusion process within a no-action region, cf.\  Karatzas and Shreve \cite{KaratzasShreve86}.
In order to explicitly construct the control (pathwise via Skorokhod's Lemma), finite fuel studies typically assume that the dynamics of the diffusion can be expressed without reference to the control (see e.g.\ \cite{Kobila93,ElKarouiKaratzas91}). 
This is different to our setup, where the non-linear mutual  interdependence between diffusion and control (local time) subverts direct construction by Skorokhod's lemma, already for OU processes \cite[Remark~1]{WardGlynn2003}.
We relate to a concrete application  in context of optimal liquidation for a financial asset position in \Cref{rmk:Opt liq problem}.

A natural idea for approximation is to proxy 'infinitesimal' reflections by small $\varepsilon$-jumps $\Delta L^\varepsilon$, thereby inducing jumps of the elastic reflection boundary, see~\cref{fig:approx diffusion in X-t-space}.
This allows to express excursion lengths of the approximating diffusion $X^\varepsilon$ in terms of independent hitting times for continuous diffusions, what naturally leads to an explicit expression \eqref{eq:Lap transf inv loc time} for the Laplace transform of the inverse local time of $X$.
In our singular control context, $L^\varepsilon$ is asymptotically optimal at first order if $L$ is optimal, see \cref{rmk:Opt liq problem}.
Our main result is \cref{thm: local time Laplace}. We prove ucp-convergence of $(X^\varepsilon, L^\varepsilon)$ to $(X,L)$ by showing  in \cref{sect: reflection approximation proof} tightness of the approximation sequence $(X^\varepsilon, L^\varepsilon)_\varepsilon$ and using Kurtz–Protter's notion of uniformly controlled variations (UCV), introduced in \cite{KurtzProtter91}.

\section{Elastic reflection: Model and notation}
\label{sect: model}
We consider a filtered probability space $\paren[\big]{\Omega,\scF,(\scF_t)_{t \ge 0},\PP}$ with one-di\-men\-sio\-nal $(\scF_t)$-Brownian motion $W$ and filtration $(\scF_t)$ satisfying the usual conditions of right-continuity and completeness.
Let $\sigma:\RR \to (0,\infty)$ and $b : \RR \to \RR$ be Lipschitz-continuous and such that the continuous $\RR$-valued $(b,\sigma)$-diffusion $\diff Z_t = b(Z_t)\diff t + \sigma(Z_t) \diff W_t$ with generator $\scG := \frac{1}{2} \sigma(x)^2\frac{\diff^2}{\diff x^2} + b(x) \frac{\diff}{\diff x}$ is regular and recurrent. 
Moreover, let $X$ be a $(b,\sigma)$-diffusion with reflection at an elastic boundary.
This means that for a given non-decreasing $g \in C^1([0,\infty))$, the processes $(X,L)$ satisfy
\begin{equation} \label{eq: diffusion}
	\diff X_t = b(X_t)\diff t + \sigma(X_t) \diff W_t - \diff L_t \,,
\qquad 
	X_0 = g(0)\,,
\end{equation}
with the reflection local time $L$ being a  continuous non-decreasing process $L$ 
that only grows when $X$ is  at the (local-time-dependent) boundary $g(L)$, i.e.\ 
\begin{equation} \label{eq: local time}
	\diff L_t = \indicator_{\braces{X_t = g(L_t)}} \diff L_t \,,
\quad 
	L_0 = 0 \,, 
\quad
\text{with }
	X_t \le g(L_t) \text{ for all } t \ge 0.
\end{equation}
Note that the reflecting boundary is not deterministic in real time and space coordinates.
Instead, the boundary $g(L)$, at which the diffusion $X$ is being reflected, is elastic in the sense that it is itself
a stochastic process which retracts when being hit, cf.\ \cref{fig:fixed boundary in X-L-space}. 
Strong existence and uniqueness of $(X,L)$ follow from classical results (cf.\ \cref{rmk: reflected-OU-difficult}) and are also an outcome of our explicit construction below, see \cref{lemma: approx weak convergence}.

We are particularly interested (see \cref{rmk:Opt liq problem}) in the inverse local time
\begin{equation} \label{def: tau}
	\tau_\ell := \inf \braces{ t > 0 \mid L_t > \ell }.
\end{equation}

\begin{remark}
Note that $\{t\geq 0\mid  X_t = g(L_t)\}$ is a.s.\ of Lebesgue measure zero by \cite[ex.~VI.1.16]{RevuzYor99}.
For a constant boundary $g(\ell) \equiv a$, Tanaka's formula for symmetric local times \cite[ex.~VI.1.25]{RevuzYor99}
hence  shows that the process $L$, that we obtain as a solution to the SDE with reflection \eqref{eq: diffusion}~--~\eqref{eq: local time}, is 
the symmetric local time of the continuous semimartingale $X$ at given level $a\in \RR$, i.e. $L_t= \lim_{\varepsilon \searrow 0} \frac{1}{2\varepsilon} \int_0^t \indicator_{(a-\varepsilon, a + \varepsilon)}(X_s) \diff \angles{X,X}_s$. 
\end{remark}

We denote by $\hittingTimeTo{y}$ the first hitting time of a point $y$ by a $(b,\sigma)$-diffusion, and  write $\hittingTimeFromTo{x}{y}$ for the hitting time when the diffusion starts in $x$.
Note that $\PP[ \hittingTimeFromTo{x}{y} < \infty ] = 1$ for all $x,y$ by our assumption on the diffusion being regular and recurrent.

\section{Approximation by small \texorpdfstring{$\varepsilon$}{\textepsilon}-reflections}
\label{sect: reflection approximation}

\begin{figure}[ht]
	\centering
	\includegraphics[height=5cm]{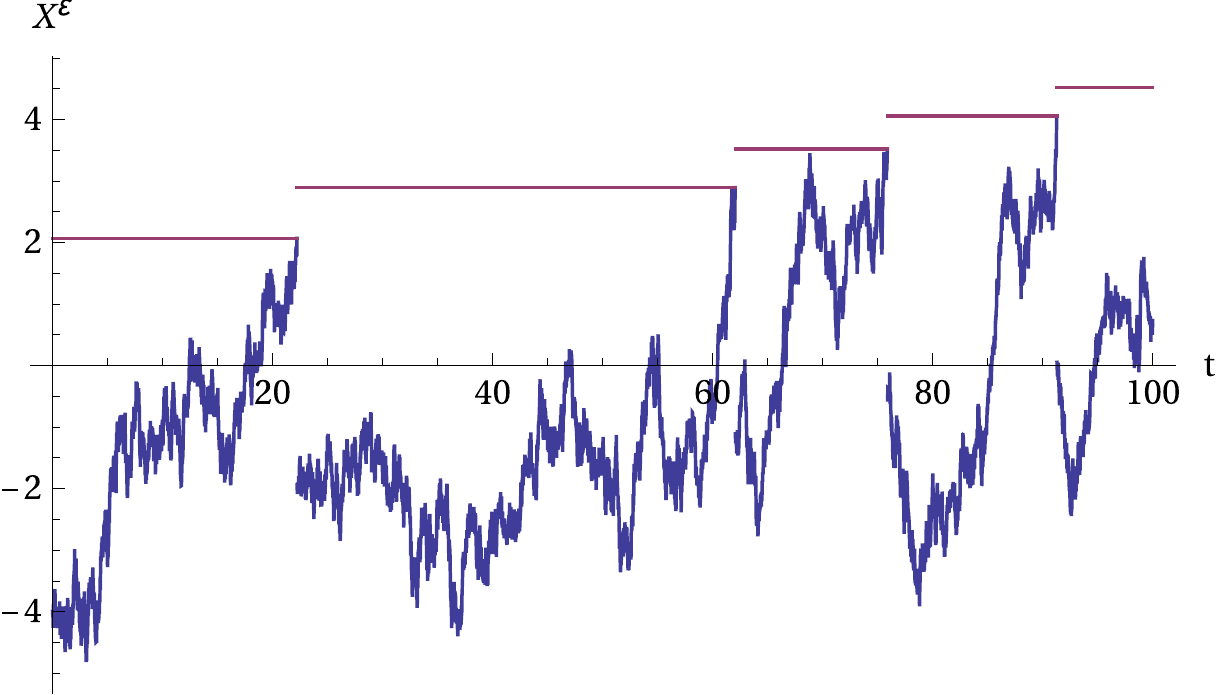}
	\caption{Approximating processes $X^\varepsilon$  and $g(L^\varepsilon)=\sqrt{L^\varepsilon}$ for $\varepsilon = 4$.}
	\label{fig:approx diffusion in X-t-space}
\end{figure}

\noindent
We construct solutions to \eqref{eq: diffusion}~--~\eqref{eq: local time} and derive 
an explicit representation 
\eqref{eq:Lap transf inv loc time} of the Laplace transform of the inverse local time at boundary $g$
 by approximating reflection by jumps in the following system of SDEs: 
\begin{align}
\label{def: approx X}
	\diff X^\varepsilon_t &\hphantom{:}= b(X^\varepsilon_t)\diff t + \sigma(X^\varepsilon_t) \diff W_t - \diff L^\varepsilon_t \,, 
	&
	X^\varepsilon_{0-} &:= g(0) \,,
\\\label{def: approx L}
	L^\varepsilon_t &:= \sum_{0 \le s \le t} \Delta L^\varepsilon_s
	\qquad\text{with }
	\Delta L^\varepsilon_t := \begin{cases}
			\varepsilon & \text{if } X^\varepsilon_{t-} = g(L^\varepsilon_{t-}),
		\\	0 & \text{otherwise},
		\end{cases}
	&
	L^\varepsilon_{0-} &:= 0 \,,
\\\label{def: approx tau}
	\tau^\varepsilon_\ell &:= \inf \braces{ t > 0 \mid L^\varepsilon_t > \ell } \quad \text{for $\ell\ge 0$}. 
\end{align}
As soon as process $X^\varepsilon$ hits the boundary, it is 
reflected  by a jump of fixed size $\varepsilon>0$.
We will speak of $L^\varepsilon$ as discrete local time, as  it is approximating  $L$ in the sense of  \cref{thm: local time Laplace}. 
Since the target reflected diffusion $X$ starts at the boundary $g$, we now have $X^\varepsilon_0=g(0)-\varepsilon$ after an initial jump $\Delta L_0^\varepsilon=\varepsilon$ away from $X^\varepsilon_{0-}:=g(0)$. 

\begin{lemma} \label{lemma: reflected approx SDE well-posed}
	For any $\varepsilon > 0$, the SDE \eqref{def: approx X}--\eqref{def: approx L} has a unique (up to indistinguishability) strong global solution $(X^\varepsilon_t,L^\varepsilon_t)_{t\ge0}$. Moreover, uniqueness in law holds.
\end{lemma}
\begin{proof}
Indeed, one can argue by results \cite[V.9--11, V.17]{RogersWilliams87vol2} for classical diffusion SDEs with Lipschitz coefficients $(b,\sigma)$ by inductive construction on $\rightOpenStochasticInterval{ 0,\tau_n }$ where for $n\ge 1$,
\(
	\tau_n:= \inf\braces{t>\tau_{n-1}\mid  X^\varepsilon_{t-}=g(n\varepsilon) } = \tau^\varepsilon_{\varepsilon n}
\) 
with \(\tau_0:=0\).
Clearly $L^\varepsilon_t$ equals $L^\varepsilon_{\tau_{n-1}}$ for $t\in \rightOpenStochasticInterval{ \tau_{n-1}, \tau_n }$ and $L^\varepsilon_{\tau_{n}} \!=\! L^\varepsilon_{\tau_{n-1}} + \varepsilon$, while $X^\varepsilon_u = F(X^\varepsilon_{\tau_{n-1}}, (W_{\tau_{n-1}+s})_{s\ge 0})_{u - \tau_{n-1}}$ on $\rightOpenStochasticInterval{ \tau_{n-1},\tau_n }$ holds for a suitable functional representation $F$ of strong solutions to $(b,\sigma)$-diffusions \cite[Theorem~V.10.4]{RogersWilliams87vol2}.
Such construction extends to $\rightOpenStochasticInterval{ 0, \tau_\infty }$ for
$\tau_\infty:=\lim_n\tau_n$.

It suffices to show $\tau_\infty=\infty$ (a.s.).
To this end, let $g_\infty:= \lim_n g(n\varepsilon)\in \mathbb{R}\cup\{\infty\}$.
In the case $g_\infty<\infty$ , one can find $x,y\in \mathbb{R}$ with $g_\infty - \varepsilon < x < y < g_{\infty}$. 
By recurrence of $(b,\sigma)$-diffusions, we have (a.s.)\ finite times
$\tau_0^y:= \inf\{t > 0 \mid X^\varepsilon_t = y \}$,
$\tau_n^x:= \inf\{t > \tau_{n-1}^y \mid X^\varepsilon_t = x \}$,
$\tau_n^y:= \inf\{t > \tau_n^x \mid X^\varepsilon_t = y \}$, for $n \in \NN$.
The durations $\tau^y_n-\tau^x_n$, $n\in \NN$, for upcrossings of the interval $[x,y]$ are i.i.d., by the strong Markov property of the time-homogeneous diffusion. Moreover, $X^\varepsilon$ is continuous on all $\closedStochasticInterval{\tau^x_n, \tau^y_n}$.
By the law of large numbers,
\( \frac 1 n \sum_{i=1}^n \exp (- \lambda(\tau^y_i-\tau^x_i) )\) 
converges almost surely for $n\to\infty$ to  
the Laplace transform $\mathbb{E}_x[\exp(-  \lambda \hittingTimeTo{y})]$, $\lambda\ge 0$, 
 of the time $\hittingTimeTo{y}$ for hitting $y$ by the $(b,\sigma)$-diffusion process (started at $x$).
This expectation is strictly less than $1$ for $\lambda>0$, as $\hittingTimeTo{y} > 0$ $P_x$-a.s.\ for $y>x$,
whereas the limit of 
\( \frac 1 n \sum_{i=1}^n \exp (- \lambda(\tau^y_i-\tau^x_i) )\) 
equals $1$ on $\{\tau_\infty < \infty\}$, where $\lim_{i \to \infty} (\tau^y_i-\tau^x_i)=0$.
Hence $P[\tau_\infty<\infty]=0$.

If $g_\infty = \infty$,
let 
\(
	\tau_n':= \inf\braces{t>\tau_{n-1} \mid X^\varepsilon_{t-} \!=\! g((n-1)\varepsilon)}
\), for $n\ge 1$,
so that $\tau_{n-1} < \tau_n' \le \tau_{n}$ 
and $X^\varepsilon_{\tau_{n}'-} = g((n-1)\varepsilon)=X^{\varepsilon}_{({\tau_{n-1}})-}$.
Using the time change $\varphi_t:=\int_0^t \sum_{n=1}^\infty 1_{\rightOpenStochasticInterval{ \tau_n',\tau_n }} \diff u$ with 
inverse $s_t:=\inf\{u\mid \varphi_u>t\}$, we get  (cf.\ \cite[IV.30.10]{RogersWilliams87vol2}) that $X_t':= X^\varepsilon_{s_t}$, $t\ge0$, solves the SDE $\!\diff X'_t = b(X'_t) \diff t + \sigma(X'_t) \diff W'_t$, $X'_0=g(0)$, on $\rightOpenStochasticInterval{0,\varphi_\infty}$ for $\varphi_\infty:= \sup_t \varphi_t$, with respect to 
$W'_t=\int_0^{s_t} \sum_{n=1}^\infty 1_{\rightOpenStochasticInterval{ \tau_n', \tau_n }} \diff W_u$.  
We have $W'_t = B_{t \wedge \varphi_\infty}$ for some Brownian motion $B$ on $[0,\infty)$ 
by the Dambis-Dubins-Schwarz theorem, 
cf.\  \cite[Thm.~3.4.6, Prob.~3.4.7]{KaratzasShreveBook}.
So $X'$ solves the $(b,\sigma)$-diffusion SDE w.r.t.\ $B$ on $\rightOpenStochasticInterval{0, \varphi_\infty}$.
Consider a $(b,\sigma)$-diffusion $\tilde X$ w.r.t.\ $B$ on $[0,\infty)$. 
By the usual Gronwall argument for uniqueness of SDE solutions, we get $X' = \tilde X$ on all $\closedStochasticInterval{0, \varphi_{\tau_n}}$
and hence $X' = \tilde X$ on $\rightOpenStochasticInterval{0, \varphi_\infty}$.
In particular, $X'$ remains a.s.\ bounded on any finite time interval $\rightOpenStochasticInterval{0,T}$ with $T \le \varphi_\infty$.
However, in the event $\{\tau_\infty < \infty\} \subset \{\varphi_\infty < \infty\}$, we get from $X'_{\varphi_{\tau_n}} = g(n\varepsilon) \to \infty$ 
that $\sup_{t < \varphi_\infty} X'_t = \infty$.
Hence, we must have $\PP[\tau_\infty < \infty] = 0$.
\end{proof}

By \eqref{def: approx X} -- \eqref{def: approx tau}, we have  $\tau^\varepsilon_{0}=\tau^\varepsilon_{0-}=0$ and $\tau^\varepsilon_{\ell}=\tau^\varepsilon_{(k-1)\varepsilon}$ for $\ell\in [(k-1)\varepsilon,k\varepsilon)$ with $k\in \NN$, and $\tau^\varepsilon_{k\varepsilon}$ is the $k$-th jump time of $X^{\varepsilon}$ and $L^{\varepsilon}$  within period $(0,\infty)$.
For $\ell=k\varepsilon$, the approximating process $X^\varepsilon$ is a continuous $(b,\sigma)$-diffusion on stochastic intervals $\rightOpenStochasticInterval{ \tau^\varepsilon_{\ell-} , \tau^\varepsilon_\ell }$, and $X^\varepsilon_{\tau^\varepsilon_{\ell}} = X^\varepsilon_{\tau^\varepsilon_{(\ell-)}} - \varepsilon =g\paren{ L^\varepsilon_{\tau^\varepsilon_{(\ell-)}} } - \varepsilon = g(\ell - \varepsilon) - \varepsilon$.
For such $\ell=k\varepsilon$, we shall call $\tau^\varepsilon_\ell - \tau^\varepsilon_{\ell-}$ the  length of the ($k$-th) excursion of $X^\varepsilon$ away from the boundary.
Note that this excursion length is independent of 
$\mathcal{F}^\varepsilon_{\tau^\varepsilon_{(\ell-)}}$ and its (conditional) distribution is
\begin{equation}\label{eq:excursion length}
	\tau^\varepsilon_\ell - \tau^\varepsilon_{\ell-} \sim \hittingTimeTo{g(\ell)} \quad\text{ under }\PP_{g(\ell - \varepsilon) - \varepsilon}\,,
\end{equation}
 what is also denoted as $\tau^\varepsilon_\ell - \tau^\varepsilon_{\ell-} \stackrel{d}= \hittingTimeFromTo{g(\ell - \varepsilon) - \varepsilon}{g(\ell)}$.
The Laplace transform of first hitting times $\hittingTimeFromTo{x}{z}$ is well-known, see  e.g.~\cite[V.50]{RogersWilliams87vol2}:
for $x,z \in \RR$ and $\lambda > 0$,
\begin{equation}\label{eq: Lap transf hitting time}
	\EE\brackets[\big]{ e^{-\lambda \hittingTimeFromTo{x}{z}} } \equiv \EE_x\brackets[\big]{ e^{-\lambda \hittingTimeTo{z}} } 
		= \begin{cases}
			\Phi_{\lambda,-}(x) / \Phi_{\lambda,-}(z) & \text{if }x < z,
		\\	\Phi_{\lambda,+}(x) / \Phi_{\lambda,+}(z) & \text{if }x > z,
		\end{cases}
\end{equation}
where functions $\Phi_{\lambda,\pm}$ are uniquely determined up to a constant factor as the increasing ($\Phi_{\lambda,-}$) respectively decreasing ($\Phi_{\lambda,+}$) positive solutions $\Phi$ of the differential equation
\(
	\scG \Phi = \lambda \Phi
\)
with generator $\scG = \frac{1}{2} \sigma(x)^2\frac{\diff^2}{\diff x^2} + b(x) \frac{\diff}{\diff x}$ of the $(b,\sigma)$-diffusion.
Since we assume the boundary function $g$ to be non-decreasing, only $\Phi_{\lambda,-}$ is of interest for our purpose.

Due to  independence of Brownian increments over disjoint time intervals, the Laplace transform of the inverse local time can be calculated  from a sum of (independent) excursion lengths at (discrete) local times $\ell_n:= \varepsilon n$ as
\begin{align}
\notag
	\EE\brackets[\big]{ \exp\paren[\big]{-\lambda \tau^\varepsilon_\ell} } 
		&= \EE\brackets[\bigg]{ \exp\paren[\bigg]{ -\lambda \sum_{n=1}^{\floor{\ell / \varepsilon}} \paren[\big]{ \tau^\varepsilon_{\ell_n} - \tau^\varepsilon_{\ell_n -} } } }
= \prod_{n=1}^{\floor{\ell/\varepsilon}} \EE\brackets[\Big]{ \exp\paren[\Big]{-\lambda \paren[\big]{ \tau^\varepsilon_{\ell_n} - \tau^\varepsilon_{\ell_n -} } } }
\\\notag
		&= \prod_{n=1}^{\floor{\ell/\varepsilon}} \EE_{g(\ell_n - \varepsilon) - \varepsilon}\brackets[\big]{ \exp\paren[\big]{-\lambda \hittingTimeTo{g(\ell_n)}} }
= \prod_{n=1}^{\floor{\ell/\varepsilon}} \frac{ \Phi_{\lambda,-}\paren[\big]{ g(\ell_n - \varepsilon) - \varepsilon } }{ \Phi_{\lambda,-}\paren[\big]{ g(\ell_n) } } 
\\\label{eq: Laplace tau as sum of logarithms}
		&= \exp\paren[\Bigg]{ \sum_{n=1}^{\floor{\ell/\varepsilon}} \log\paren[\bigg]{ \frac{ \Phi_{\lambda,-}\paren[\big]{ g(\ell_n - \varepsilon) - \varepsilon } }{ \Phi_{\lambda,-}\paren[\big]{ g(\ell_n) } } } },
\end{align}
for $\ell \ge 0$ and $\lambda > 0$.
With $h_n(\xi):= \Phi_{\lambda,-}\paren[\big]{ g(\ell_n - \xi) - \xi }$, each summand in~\eqref{eq: Laplace tau as sum of logarithms} equals
\begin{align}
\notag
	\log h_n(\varepsilon) - \log h_n(0) 
		&= \int_0^\varepsilon \frac{h_n'(\xi)}{h_n(\xi)} \diff\xi 
		= -\int_0^\varepsilon \paren[\big]{ g'(\ell_n-\xi) + 1 } \frac{ \Phi_{\lambda,-}'\paren[\big]{ g(\ell_n - \xi) - \xi } }{ \Phi_{\lambda,-}\paren[\big]{g(\ell_n - \xi) - \xi} } \diff\xi 
\\
		&= -\int_{\ell_{n-1}}^{\ell_n} \paren[\big]{ g'(a) + 1 } \frac{ \Phi_{\lambda,-}'\paren[\big]{ g(a) + a - \ell_n } }{ \Phi_{\lambda,-}\paren[\big]{ g(a) + a - \ell_n } } \diff a \,.
\end{align}
Therefore, we obtain
\begin{equation} \label{eq: Laplace tau epsilon}
	\EE\brackets[\big]{ \exp\paren[\big]{ -\lambda \tau^\varepsilon_\ell } } 
		= \exp\paren[\bigg]{ -\int_0^{\varepsilon \floor{ \ell/\varepsilon }} \paren[\big]{ g'(a) + 1 } \frac{ \Phi_{\lambda,-}'\paren[\big]{ g(a) + a - \varepsilon\ceil{ a/\varepsilon } } }{ \Phi_{\lambda,-}\paren[\big]{ g(a) + a - \varepsilon\ceil{ a/\varepsilon } } } \diff a }.
\end{equation}
Intuitively, this already suggests the formula (\refeq{eq:Lap transf inv loc time}) when taking $\varepsilon \to 0$.
\begin{theorem} \label{thm: local time Laplace}
	The approximations $(X_t^\varepsilon, L_t^\varepsilon)_{t\ge 0}$ from \eqref{def: approx X}--\eqref{def: approx L} converge  uniformly in probability for $\varepsilon \to 0$ to a pair $(X_t,L_t)_{t\ge 0}$  of continuous adapted processes with non-decreasing $L$,
which is the unique strong  solution (globally on $[0,\infty)$) to the reflected SDE \eqref{eq: diffusion}--\eqref{eq: local time}.
	The inverse local time $\tau_\ell:= \inf\braces{ t>0 \mid L_t > \ell }$ has the Laplace transform
	\begin{equation}\label{eq:Lap transf inv loc time}
		\EE\brackets[\big]{ e^{-\lambda \tau_\ell} } = \exp\paren[\bigg]{ -\int_0^\ell \paren[\big]{ g'(a) + 1 } \frac{ \Phi_{\lambda,-}'\paren[\big]{ g(a) } }{ \Phi_{\lambda,-}\paren[\big]{ g(a) } } \diff a } \quad \text{for $\lambda > 0$, $\ell \ge 0$},
	\end{equation}
	 where $\Phi_{\lambda,-}$ is the (up to a constant factor) unique positive increasing solution of the differential equation $\scG \Phi = \lambda \Phi$, for $\scG$ denoting the generator of the $(b,\sigma)$-diffusion.
\end{theorem}
\begin{proof}
Existence and uniqueness of $(X,L)$ is shown in \cref{lemma: approx weak convergence} below.
\Cref{lemma: approx ucp} gives uniform convergence in probability.
Using dominated convergence for the right-hand side of \cref{eq: Laplace tau epsilon}, we find
\(
	\lim_{\varepsilon \to 0} \EE\brackets[]{ e^{-\lambda \tau^\varepsilon_\ell} } = \exp\paren[\big]{ -\int_0^\ell \paren[]{ g'(a) + 1 } \frac{ \Phi_{\lambda,-}'\paren[]{ g(a) } }{ \Phi_{\lambda,-}\paren[]{ g(a) } } \diff a }.
\)
For the left-hand side, it suffices to prove weak convergence $\tau^\varepsilon_\ell \weakConvergenceTo \tau_\ell$ as $\varepsilon \to 0$ for all $\ell \ge 0$.
This is done in \cref{lemma: tau weak convergence} below.
\end{proof}

\begin{remark}\label{rmk: reflected-OU-difficult}
	Existence and uniqueness for $(X,L)$ 
 can also be concluded from classical results, cf.\ \cite[suitably extended to non-bounded domains]{DupuisIshii93},
 by considering the pair $(X,L)$ as a degenerate diffusion in $\RR^2$ with oblique  reflection in direction $(-1, +1)$ at a smooth boundary, see~\cref{fig:fixed boundary in X-L-space}. This uses an iteration argument involving the Skohorod-map and
yields another approximation by a sequence of continuous processes. Yet, these do not satisfy the target diffusive dynamics inside the domain, except at the limiting fixed point (unless $(b,\sigma)$ are constant).
 In contrast, $(X^\varepsilon, L^\varepsilon)$ adheres to the same dynamics as $(X,L)$ between jump times, cf.\ (\ref{eq: diffusion}) and (\ref{def: approx X}), 
is Markovian and has a natural interpretation. 
\end{remark}

\begin{remark}
\label{rmk:Opt liq problem}
	\newcommand{\baseS}{\bar S}
	\newcommand{\impactVolatility}{\hat\sigma}
	\newcommand{\correlationBW}{\rho}
	\newcommand{\BigO}{\mathcal{O}}
	An application example for \eqref{eq:Lap transf inv loc time} and  elastically reflected diffusions is the optimal execution for the sale of a financial asset position if liquidity is stochastic, see \cite{BechererBilarevFrentrup2018}.
	A large trader with adverse price impact  seeks to maximize expected proceeds from selling $\theta$ risky  assets in an illiquid market.
	His trading strategy $A$ (predictable, c\`adl\`ag, non-decreasing) affects the asset price
 $S_t = f(Y^A_t)\baseS_t$ via a volume impact process $\diff Y^A_t = -\beta Y^A_t \diff t + \impactVolatility \diff B_t - \diff A_t$ with 
	$\baseS_t = \scE(\sigma W)_t$ for an increasing function $f$, and Brownian motions $(B,W)$ with correlation $\correlationBW$.
	The gains to maximize in expectation are 
	\[
		G_T(A):= \int_0^T e^{-\delta t} f(Y^A_t) \baseS_t \diff A^c_t + \sum_{\substack{ 0\le t\le T \\ \Delta A_t\ne 0 }} e^{-\delta t} \baseS_t \int_0^{\Delta A_t} f(Y^A_{t-} - x) \diff x
		.
	\]
	The optimal strategy turns out to be the local time $L$ of a reflected Ornstein-Uhlenbeck process $X$ (with $b(x):= \correlationBW\sigma\impactVolatility-\beta x$ and  $\sigma(x) = \sigma >0$) at a suitable elastic boundary $g$, as in \eqref{eq: diffusion}--\eqref{eq: local time}, see \cite[Section~3]{BechererBilarevFrentrup2018}.
	After a change of measure argument, one can write the expected proceeds from such strategies as $\EE[G_\infty(L)] = \int_0^\theta f\paren[\big]{ g(\ell) } \EE\brackets[\big]{ e^{-\delta\tau_\ell} } \diff \ell$. To find the optimal free boundary $g$, one can then apply \eqref{eq:Lap transf inv loc time} to express the proceeds as a functional of the boundary $g$, and optimize over all possible boundaries by solving a calculus of variations problem. 
	This is key to the proof in \cite{BechererBilarevFrentrup2018}.
	The discrete local time $L^\varepsilon$ has a natural interpretation as step process which approximates the continuous optimal strategy $L$  by doing small block trades, as they would be realistic in an actual implementation, with identical (no-)action region.
	The approximation is  asymptotically optimal for the control problem. Indeed, straightforward calculations similar to the derivation of \eqref{eq: Laplace tau epsilon} show  that $L^\varepsilon$ is asymptotically optimal in first order, i.e.\ $\EE[G_\infty(L)] = \EE[G_\infty(L^\varepsilon)] + \BigO(\varepsilon)$.
\end{remark}

\section{Tightness and convergence}
\label{sect: reflection approximation proof}

To show convergence of $(\tau^\varepsilon_\ell)_{\varepsilon}$, we will prove that the pair of càdlàg processes $(X^\varepsilon, L^\varepsilon)$ forms a tight sequence in $\varepsilon\rightarrow 0$. Applying weak convergence theory for SDEs by Kurtz and Protter \cite{KurtzProtter96}, we show that any limit point (for $\varepsilon\rightarrow 0$) satisfies \eqref{eq: diffusion} and \eqref{eq: local time}. Uniqueness in law for solutions of \eqref{eq: diffusion} -- \eqref{eq: local time} will then allow to conclude \cref{thm: local time Laplace}.

Let $(\varepsilon_n)_{n\in\NN}$ be a sequence with $\varepsilon_n \to 0$ and consider the sequence $(X^{\varepsilon_n}, L^{\varepsilon_n})_n$.
To show tightness, we will apply the following
criterion due to Aldous.
\begin{proposition}[{%
\cite[Cor.\ to Thm.~16.10]{Billingsley99}%
}] \label{prop: Aldous tightness criterion}
	Let $(E,\abs{\cdot})$ be a separable Banach space.
	If a sequence $(Y^n)_{n \in \NN}$ of adapted, $E$-valued càdlàg processes satisfies the following two conditions, then it is tight.
	\begin{enumerate}[label=(\alph*)]
		\item \label{prop: Aldous tightness criterion -- part: tight jumps}
			The sequences $\paren[\big]{ J_T(Y^n) }_n$ and $\paren{ Y^n_0 }_n$ are tight (in $\RR$, resp.\ $E$) for any $T \in (0,\infty)$, with
\(
				J_T(Y^n) := \sup_{0 < t \le T} \abs[\big]{ Y^n_t - Y^n_{t-} }\,
			\)		
denoting the largest jump until time $T$.
			
		\item \label{prop: Aldous tightness criterion -- part: jump probability}
			For any $T \in (0,\infty)$ and $\varepsilon_0, \eta > 0$ there exist $\delta_0 > 0$ and $n_0 \in \NN$ such that for all $n \ge n_0$, all (discrete) $Y^n$-stopping times $\aldousStoppingTime \le T$ and all $\delta \in (0, \delta_0]$ we have
			\[
				\PP\brackets[\big]{ \abs{Y^n_{\aldousStoppingTime+\delta} - Y^n_\aldousStoppingTime} \ge \varepsilon_0 } \le \eta\,.
			\]

	\end{enumerate}
\end{proposition}

\noindent
To get tightness one needs to control both jump size and, regarding $(L^{\varepsilon}_n)_n$, the frequency of jumps simultaneously.
As we are considering processes with jumps of size $\pm \varepsilon_n\to 0$, so only the latter is not yet clear.
To this end, the next lemma provides a technical bound on $X^{\varepsilon_n}$, $L^{\varepsilon_n}$, while a second lemma constricts the probability that $X^{\varepsilon_n}$ (respectively $L^{\varepsilon_n}$) performs a number of $N_n$ jumps in a time interval of fixed length.

\begin{lemma}[Upper bound] \label{lemma: boundary upper bound}
	Fix a time horizon $T \in (0,\infty)$ and $\eta \in(0,1]$.
	Then there exists a constant $M \in \RR$ such that
	\(
		\PP\brackets{ \exists n : g(L^{\varepsilon_n}_T - \varepsilon_n) > M } \le \eta
	\),
	 with the domain of definition for the function $g$ being extended by $g(-x):= g(0)$ for $-x < 0$.
\end{lemma}
\begin{proof}

Consider a continuous $(b,\sigma)$-diffusion $Y$ that starts at time $t=0$ at $g(0)$. For $n\in \NN$ and $k=0,1,2,\ldots$, let  $\ell(n,k):= k\varepsilon_n$. By induction over $k$, using comparison for diffusion SDEs, cf.~\cite[Theorem~5.2.18]{KaratzasShreveBook},  one obtains that (a.s.) $X^{\varepsilon_n}_t\leq Y_t$ for $t\in [0, \tau^{\varepsilon_n}_{\ell(n,k)})$ for all $k\geq 1$, and hence $X^{\varepsilon_n} \leq Y$ on $[0,\infty)$ (a.s.) because $\lim_{k\rightarrow \infty} \tau^{\varepsilon_n}_{\ell(n,k)} = \infty$ for any $n$ by \cref{lemma: reflected approx SDE well-posed}. Hence, on the event $\{\exists n : g(L^{\varepsilon_n}_T - \varepsilon_n) > M\}$ we have $\sup_{t\in [0,T]}Y_{t} \geq M$, and hence $\hittingTimeFromTo{g(0)}{M} \le T$. Thus   
\(
		\PP\brackets{\exists n : g(L^{\varepsilon_n}_T - \varepsilon_n) > M} \le \PP\brackets[]{ \hittingTimeFromTo{g(0)}{M} \le T }\,.
\)
Now the claim follows since $\lim_{M \to \infty} \PP[\hittingTimeFromTo{g(0)}{M} \le T] = 0$.
\end{proof}

\begin{lemma}[Frequency of jumps] \label{lemma: jump frequency}
	Fix $T \in (0,\infty)$, $\varepsilon_0,\eta > 0$, and set $N_n:= \ceil{\varepsilon_0 / \varepsilon_n}$.
	Then there exists $\delta > 0$ and $n_0\in \NN$ such that for every bounded stopping time $\aldousStoppingTime \le T$ we have
	\(
		\PP\brackets[\big]{ J^{\varepsilon_n}_{\aldousStoppingTime,\delta} \ge N_n } \le \eta
	\)
	for all $n \ge n_0$, where $J^{\varepsilon_n}_{\aldousStoppingTime, \delta}:= \inf\{ k \mid L^{\varepsilon_n}_{\aldousStoppingTime} + k \varepsilon_n \ge L^{\varepsilon_n}_{\aldousStoppingTime + \delta} \}$ is the number of jumps of $X^{\varepsilon_n}$, respectively $L^{\varepsilon_n}$, in time $\leftOpenStochasticInterval{ \aldousStoppingTime,\aldousStoppingTime + \delta }$.
\end{lemma}
\begin{proof}
	We will first find an estimate for the jump count probability for arbitrary but fixed $\delta>0$, $n\in \NN$, $N_n\in \NN$  and $\aldousStoppingTime \le T$. 
	Only in part 2) of the proof we will consider $(N_n)_{n\in \NN}$ as stated, to study the limit $n \to \infty$.
	More precisely, we will show in part 1) that, given $\scF_\aldousStoppingTime$, for every $\lambda>0$ there exist $k_{n,\lambda} \in \braces{ 0, 1, \ldots, N_n - 1 }$ s.t.\ for $x_n:= g\paren{L^{\varepsilon_n}_\aldousStoppingTime + \varepsilon_n k_{n,\lambda}}$,
	\begin{equation}\label{ineq: jump frequency -- fixed n}
		\PP\brackets[\big]{ J^{\varepsilon_n}_{\aldousStoppingTime, \delta} \ge N_n \bigm| \scF_\aldousStoppingTime } \le e^{\lambda \delta} \paren[\bigg]{ \frac{ \Phi_{\lambda,-}(x_n-\varepsilon_n) }{ \Phi_{\lambda,-}(x_n) } }^{N_n-1}
		.
	\end{equation}
	
	1) In this part, fix arbitrary $\delta > 0$, $n \in \NN$, $N_n\in \NN$ and $\hat \tau \le T$.
	We enumerate the jumps and estimate the sum of excursion lengths by $\delta$.
	Let $\ell_k:= L^{\varepsilon_n}_\aldousStoppingTime + k\varepsilon_n$ be the (discrete) local time at the $k$-th jump after time $\aldousStoppingTime$.
	If $X^{\varepsilon_n}$ has at least $N_n$ jumps in the interval $\leftOpenStochasticInterval{ \aldousStoppingTime, \aldousStoppingTime + \delta }$, it is doing at least $N_n-1$ complete excursions (cf.\ \eqref{eq:excursion length}),
 so that,  noting that $\tau^{\varepsilon_n}_{L^{\varepsilon_n}_{t}-\varepsilon_n}\le t <  \tau^{\varepsilon_n}_{L^{\varepsilon_n}_{t}}$ (for all $t\ge 0$) and $\ell_{N_n - 1} + \varepsilon_n \le L^{\varepsilon_n}_{\aldousStoppingTime + \delta}$,
we have 
	\[
		\delta 
			= (\aldousStoppingTime + \delta) - \aldousStoppingTime 
			\ge \tau^{\varepsilon_n}_{L^{\varepsilon_n}_{\aldousStoppingTime+\delta}-\varepsilon_n } - \tau^{\varepsilon_n}_{L^{\varepsilon_n}_\aldousStoppingTime}
			\ge \sum_{k=1}^{N_n-1} \paren[\big]{ \tau^{\varepsilon_n}_{\ell_k} - \tau^{\varepsilon_n}_{\ell_{k-1}} }
			\stackrel{d}= \sum_{k=1}^{N_n-1} H_k
	\]
	 with the last equality
being in distribution conditionally on $\mathcal{F}_{\aldousStoppingTime}$,
 for $H_k$ being conditionally independent and distributed as $ \hittingTimeFromTo{g(\ell_{k-1})-\varepsilon_n}{g(\ell_k)}$.
Clearly, $\ell_k$ is $\mathcal{F}_{\aldousStoppingTime}$-measurable.
	By the Laplace transform \eqref{eq: Lap transf hitting time} of $H_k$ and  the Markov inequality, we get for $\lambda > 0$ 
	\begin{align*}
		\MoveEqLeft
		\PP\brackets[\big]{ J^{\varepsilon_n}_{\aldousStoppingTime,\delta} \ge N_n \bigm| \scF_{\aldousStoppingTime}}
			\le \PP\brackets[\bigg]{ \sum_{k=1}^{N_n - 1} H_k \le \delta \biggm| \scF_{\aldousStoppingTime} }
			\le e^{\lambda\delta} \EE\brackets[\bigg]{ \exp\paren[\bigg]{ -\lambda \sum_{k=1}^{N_n - 1} H_k } \biggm| \scF_{\aldousStoppingTime} }
		\\
			&= e^{\lambda\delta} \prod_{k=1}^{N_n - 1} \EE\brackets[\Big]{ \exp\paren[\Big]{-\lambda \hittingTimeFromTo{g(\ell_{k-1}) - \varepsilon_n}{g(\ell_k)}} \Bigm| \scF_{\aldousStoppingTime} }
		\\ \displaybreak[0]
			&= e^{\lambda\delta} \prod_{k=1}^{N_n - 1} \frac{ \Phi_{\lambda,-}\paren[\big]{ g(\ell_{k-1}) - \varepsilon_n } }{ \Phi_{\lambda,-}\paren[\big]{ g(\ell_k) } }
			\le e^{\lambda \delta} \prod_{k=1}^{N_n - 1} \frac{ \Phi_{\lambda,-}\paren[\big]{ g(\ell_k) - \varepsilon_n } }{ \Phi_{\lambda,-}\paren[\big]{ g(\ell_k) } }
		\\
			&\le e^{\lambda\delta} \paren[\bigg]{ \max_{0 \le k < N_n} \frac{ \Phi_{\lambda,-}\paren[\big]{ g(\ell_k) - \varepsilon_n } }{ \Phi_{\lambda,-}\paren[\big]{ g(\ell_k) } } }^{N_n - 1}
			= \;e^{\lambda\delta} \paren[\bigg]{ \frac{ \Phi_{\lambda,-}\paren{ x_n - \varepsilon_n } }{ \Phi_{\lambda,-}\paren{ x_n } } }^{N_n - 1}
	\end{align*}
	where $x_n:= g(\ell_k)$ for the index $k = k_{n,\lambda}$ attaining the maximum.

	2) For given $\delta > 0$ and $\aldousStoppingTime \le T$, let us now consider the sequence $N_n= \ceil{\varepsilon_0 / \varepsilon_n}$, $n\in \NN$.
	To investigate the limit $n \to \infty$, first observe that by Taylor expansion
	\[
		\log \frac{ \Phi_{\lambda,-}(x - \varepsilon_n) }{ \Phi_{\lambda,-}(x) } = -\varepsilon_n \frac{ \Phi_{\lambda,-}'(x) }{ \Phi_{\lambda,-}(x) } + \varepsilon_n r(x,\varepsilon_n),
	\]
	where $r(\cdot,\varepsilon_n) \to 0$ converges uniformly on compacts for $\varepsilon_n \to 0$.
	Since $\aldousStoppingTime + \delta \leq T+\delta$ is bounded, \cref{lemma: boundary upper bound} yields a constant $M \in \RR$ such that
	\(
		\PP\brackets[\big]{ \exists n : x_n > M } \le \frac{\eta}{2}
	\)
	for the $x_n$ from above.
	On the event $\{ \forall n: x_n \in I \}$ with compact $I := [g(0), M]$, we have  uniform convergence of $r(x_n, \varepsilon_n)$ and thereby get 
	\begin{align*}
		\MoveEqLeft
		\hspace{-1em}
		\limsup_{n\to\infty} e^{\lambda\delta} \paren[\bigg]{ \frac{ \Phi_{\lambda,-}(x_n - \varepsilon_n) }{ \Phi_{\lambda,-}(x_n) } }^{N_n - 1}
		= \exp\paren[\bigg]{ \lambda\delta + \limsup_{n \to \infty}\, (N_n - 1)\log \frac{ \Phi_{\lambda,-}(x_n - \varepsilon_n) }{ \Phi_{\lambda,-}(x_n) } }
	\\
		&= \exp\paren[\bigg]{ \lambda\delta + \limsup_{n \to \infty}\, (N_n\varepsilon_n - \varepsilon_n) \paren[\Big]{ r(x_n, \varepsilon_n) - \frac{ \Phi_{\lambda,-}'(x_n) }{ \Phi_{\lambda,-}(x_n) } } }
	\\
		&\le \exp\paren[\bigg]{ \lambda\delta - \varepsilon_0 \inf_{x \in I} \frac{ \Phi_{\lambda,-}'(x) }{ \Phi_{\lambda,-}(x) } }
		= \sup_{x \in I} \exp\paren[\bigg]{\lambda\delta - \varepsilon_0 \frac{ \Phi_{\lambda,-}'(x) }{ \Phi_{\lambda,-}(x) } } .
	\end{align*} 
	By \cite[Theorem~1]{PitmanYor03},   $\psi^x(\lambda) := \frac{1}{2} \Phi_{\lambda,-}'(x) / \Phi_{\lambda,-}(x)$ is the Laplace exponent of $A^x(\kappa^x_\cdot)$, where $\kappa^x_\ell$ is the inverse local time at constant level $x$ of a $(b,\sigma)$-diffusion $Z^x$ starting at $x$,  and $A^x(t)$ is the occupation time 
	\(
		A^x(t) := \int_0^t \indicator_{\braces{ Z^x_s \le x }} \diff s\,.
	\)
	So we get for $\lambda \to \infty$ that $\exp\paren[\big]{-2 \varepsilon_0 \psi^x(\lambda)} = \EE_x\brackets[\big]{ \exp\paren[\big]{ -\lambda A^x(\kappa^x_{2 \varepsilon_0}) } } \to 0$.
	By compactness of $I$ and Dini's theorem there exists $\lambda = \lambda_{\varepsilon_0,\eta,M}$ such that for $\delta:= 1 / \lambda$ we have
	\begin{align}\label{eq:last step pf on number of jumps}
		\limsup_{n \to \infty} e^{\lambda\delta} \paren[\bigg]{ \frac{ \Phi_{\lambda,-}(x_n - \varepsilon_n) }{ \Phi_{\lambda,-}(x_n) } }^{N_n - 1}
			&\le e^{\lambda\delta} \sup_{x \in I} \exp\paren[\big]{ -2 \varepsilon_0 \psi^x(\lambda) } \le \frac{\eta}{2} 
	\end{align} 
	on the event $\{x_n \le M\text{ for all }n\}$. 
	By \cref{ineq: jump frequency -- fixed n} and $\PP[\exists n: x_n > M] \le \eta/2$, this completes the proof.
\end{proof}

Using the preceding two lemmas, we will first prove tightness of $(L^{\varepsilon_n})_n$ and of $(X^{\varepsilon_n})_n$  separately. 
Tightness of the pair $(X^{\varepsilon_n}, L^{\varepsilon_n})_n$ is handled afterwards.

\begin{lemma}[{Tightness of the local time approximations}]\label{lem: Le tight}
	The sequence $(L^{\varepsilon_n})_n$ of càdlàg processes defined by \labelcref{def: approx L,def: approx X} satisfies Aldous' criterion and thus is tight.
\end{lemma}
\begin{proof}
	Part~\ref{prop: Aldous tightness criterion -- part: tight jumps} of \cref{prop: Aldous tightness criterion} is clear, as the  initial value $L^{\varepsilon_n}_0 = \varepsilon_n$ is deterministic and $J_T(L^{\varepsilon_n}) \le \varepsilon_n$.
	For part~\ref{prop: Aldous tightness criterion -- part: jump probability}, consider $T,\eta,\varepsilon_0 > 0$ and any bounded $L^{\varepsilon_n}$-stopping time $\aldousStoppingTime \le T$.
	The event $\abs{L^{\varepsilon_n}_{\aldousStoppingTime+\delta} - L^{\varepsilon_n}_\aldousStoppingTime} \ge \varepsilon_0$ means that $L^{\varepsilon_n}$ performs at least $N_n := \ceil{ \varepsilon_0 / \varepsilon_n }$ jumps in the stochastic interval $\leftOpenStochasticInterval{\aldousStoppingTime, \aldousStoppingTime + \delta}$.
 \cref{lemma: jump frequency} yields some $n_0$ and $\delta_0 = \delta_0(\varepsilon_0)$ such that Aldous' criterion is satisfied for all $n \ge n_0$.
	Hence, $(L^{\varepsilon_n})_n$ is tight by \cref{prop: Aldous tightness criterion}.
\end{proof}
\noindent
Next we show boundedness of $(X^{\varepsilon_n})_n$, needed for \cref{lem: Xe tight} to prove tightness.

\begin{lemma}[Bounding the diffusion approximations] \label{lemma: Xepsilon bound}
	Let $T \in (0,\infty)$ and $\eta > 0$.
	Then there exists $M \in \RR$ such that $\PP\brackets{ \sup_{t \in [0,T]} \abs{ X^{\varepsilon_n}_t } > M } < \eta$ for all $n\in\NN$.
\end{lemma}
\begin{proof}
By \cref{lemma: boundary upper bound}, for every $n\in \NN$ the process $X^{\varepsilon_n}$ on $[0,T]$ is bounded from above by a constant $M$ with probability at least $1 -\eta/2$ . It remains to show that it is also bounded from below with high probability.
To this end, we will construct a process $Y$ that is a lower bound for all $X^{\varepsilon_n}$ and then argue for $Y$.

For $\hat\varepsilon:= \sup_n \varepsilon_n$  consider a $(b,\sigma)$-diffusion $Y$  which is discretely reflected by jumps of size $-\hat\varepsilon$ at a constant boundary $c := g(0)-\hat\varepsilon$, with $Y_0 = y:= g(0)-2\hat\varepsilon$. 
Such $Y$ is a special case of \eqref{def: approx X}--\eqref{def: approx L}, for a constant boundary function:
$\diff Y_t = b(Y_t)\diff t + \sigma(Y_t) \diff W_t - L^Y_t$ with $L^Y_t := \sum_{0\le s\le t} \Delta L^Y_t$ and $\Delta L^Y_t := \hat\varepsilon \indicator_{\{ Y_{t-} = c \}}$.
Let $\tau^Y_k:= \inf \{ t > 0 \mid L^Y_t > k\hat\varepsilon \}$ be the $k$-th hitting time of $Y$ at the boundary $c$. 
So on all $\rightOpenStochasticInterval{ \tau^Y_k, \tau^Y_{k+1} }$, $Y$ is a continuous $(b,\sigma)$-diffusion starting in $y$.
Now for fixed $n$, $\varepsilon:= \varepsilon_n$, note that $X^{\varepsilon}_{\tau^{\varepsilon}_{m \varepsilon}} = g((m-1) \varepsilon) - \varepsilon \ge c \ge Y_{\tau^{\varepsilon}_{m \varepsilon}}$ by monotonicity of $g$.
As $\tau^{\varepsilon}_{m \varepsilon} \to \infty$ for $m \to \infty$ by \cref{lemma: reflected approx SDE well-posed}, induction over the inverse (discrete) local times $\tau^{\varepsilon}_{m\varepsilon}$, $m \in \NN$, yields $X^{\varepsilon} \ge Y$ on $\closedStochasticInterval{ \tau^Y_k, \tau^Y_{k+1} }$ if $X^{\varepsilon}_{\tau^Y_k} \ge Y_{\tau^Y_k}$ by comparison results \cite[Thm.~5.2.18]{KaratzasShreveBook}. 
Since $X^{\varepsilon}_0 \ge Y_0$, the latter follows by induction over $k$.
As $\tau^Y_k \to \infty$ for $k \to \infty$ by \cref{lemma: reflected approx SDE well-posed}, we get $X^{\varepsilon_n} \ge Y$ on $[0,\infty)$ for all $n$.
So it suffices to show $\PP[ \inf_{t \in [0,T]} Y_t < -M ] < \eta/2$ for some $M$,
which directly follows from the c\`adl\`ag property of $Y$.
\end{proof}

\begin{lemma}[Tightness of the reflected diffusion approximations]\label{lem: Xe tight}
	The sequence $\paren{ X^{\varepsilon_n} }_n$ of càdlàg processes from \labelcref{def: approx L,def: approx X} satisfies Aldous' criterion and thus is tight.
\end{lemma}
\begin{proof}
Condition~\ref{prop: Aldous tightness criterion -- part: tight jumps} of \cref{prop: Aldous tightness criterion} holds.
To verify part \ref{prop: Aldous tightness criterion -- part: jump probability}, let $\eta > 0$, $T \in (0,\infty)$, and $\aldousStoppingTime \leq T$ be a stopping time.
By \cref{lemma: Xepsilon bound},   $|X^{\varepsilon_n}_{\aldousStoppingTime}|$
is  with a probability of  at least $1-\eta/4$ bounded by some constant $M$  (not depending on $n$ and $\aldousStoppingTime$).
Let us consider the events $\{ X^{\varepsilon_n}_{\aldousStoppingTime+\delta} \le X^{\varepsilon_n}_\aldousStoppingTime - \varepsilon_0 \} \cup \{ X^{\varepsilon_n}_{\aldousStoppingTime+\delta} \ge X^{\varepsilon_n}_\aldousStoppingTime + \varepsilon_0 \} = \{ \abs{ X^{\varepsilon_n}_{\aldousStoppingTime+\delta} - X^{\varepsilon_n}_{\aldousStoppingTime} } \ge \varepsilon_0 \}$ separately.

1) First consider $\{ X^{\varepsilon_n}_{\aldousStoppingTime+\delta} \le X^{\varepsilon_n}_\aldousStoppingTime - \varepsilon_0 \}$.
For $\xi:= X^{\varepsilon_n}_{\aldousStoppingTime}$  we construct a reflected process $Y^\xi$ such that  $Y^\xi_t\le X^{\varepsilon_n}_{\aldousStoppingTime+t}$ for all $t\ge 0$.
We then estimate $\PP[ X^{\varepsilon_n}_{\aldousStoppingTime+\delta} \le X^{\varepsilon_n}_{\aldousStoppingTime} - \varepsilon_0 ]$ by means of $\PP[ Y^x_\delta \le x - \varepsilon_0 ]$ in \eqref{eq:displ of X bounded by Yx}, uniformly for all $n$ large enough.
We estimate the latter in \eqref{ineq: Y escape via N downcrossings} using the probability of a down-crossing in time $\delta$ of intervals $[x-\varepsilon_0, x-2\hat\varepsilon]$ by a continuous diffusion.
Covering $\bigcup_x [x-\varepsilon_0, x-2\hat\varepsilon]$ by finitely many intervals $[y_k, y_{k+1}]$ in \eqref{ineq: X escape down probability} then allows us to choose $\delta > 0$ sufficiently small.

To this end, choose $\hat\varepsilon \le \varepsilon_0 / 4$ and $n$ large enough such that $\varepsilon_n \le \hat\varepsilon$, and let $(Y^\xi_t)_{t\ge 0}$ be the $(b,\sigma)$-diffusion w.r.t.\ the Brownian motion $(W_{\aldousStoppingTime+t} - W_{\aldousStoppingTime})_{t\geq 0}$ with $Y^\xi_0 = \xi - 2\hat\varepsilon$, which is discretely reflected by jumps of size $-\hat\varepsilon$ at a constant boundary at level $\xi - \hat\varepsilon$.
More precisely, $\diff Y^\xi_t = b(Y^\xi_t)\diff t + \sigma(Y^\xi_t)\diff W_{\aldousStoppingTime+t} - K^\xi_t$ with (discrete) local time $K^\xi_t:= \sum_{0\le s\le t} \Delta K^\xi_s$ for $\Delta K^\xi_t:= \hat\varepsilon \indicator_{\{ Y^\xi_{t-} = \xi-\hat\varepsilon \}}$. Global existence and  uniqueness of $(Y^\xi, K^\xi)$ follows from proof of \cref{lemma: reflected approx SDE well-posed}.
By comparison arguments and induction as in the proof of \cref{lemma: Xepsilon bound}, one verifies $Y^\xi_t \le X^{\varepsilon_n}_{\aldousStoppingTime + t}$ for $t \in [0,\infty)$.
Indeed, \cite[Theorem~5.2.18]{KaratzasShreveBook} gives $Y^\xi_\cdot \le X^{\varepsilon_n}_{\aldousStoppingTime + \cdot}$ on $\rightOpenStochasticInterval{0,\tau_1}$ until the first jump of either $Y^\xi_\cdot$ or $X^{\varepsilon_n}_{\aldousStoppingTime + \cdot}$ at time $\tau_1 > 0$. 
If only $Y^\xi$ jumps, we have $Y^\xi_{\tau_1} = Y^\xi_{(\tau_1)-} - \hat\varepsilon \le X^{\varepsilon_n}_{(\tau_1)-} - \hat\varepsilon = X^{\varepsilon_n}_{\tau_1}- \hat\varepsilon$, but if $X^{\varepsilon_n}_{\aldousStoppingTime + \cdot}$ jumps, we have $X^{\varepsilon_n}_{\aldousStoppingTime + \tau_1} = g(L^{\varepsilon_n}_{(\aldousStoppingTime + \tau_1)-}) - \varepsilon_n \ge g(L^{\varepsilon_n}_{\aldousStoppingTime}) - \varepsilon_n = \xi \ge Y^\xi_{\tau_1}$.
Now $Y^\xi_{\tau_1} \le X^{\varepsilon_n}_{\aldousStoppingTime + \tau_1}$, so we get $Y^\xi_\cdot \le X^{\varepsilon_n}_{\aldousStoppingTime + \cdot}$ on $\rightOpenStochasticInterval{\tau_k,\tau_{k+1}}$ by induction for all jump times $\tau_k$ of $(Y^\xi_\cdot, X^{\varepsilon_n}_{\aldousStoppingTime + \cdot})$.

Using $Y^\xi_\delta \leq X^{\varepsilon_n}_{\aldousStoppingTime + \delta}$ and the strong Markov property of $Y^\xi$ w.r.t.\ $(\mathcal{F}_{\aldousStoppingTime + t})_{t\ge 0}$, we get 
\begin{align}\label{eq:displ of X bounded by Yx}
	\PP\brackets[\big]{ X^{\varepsilon_n}_{\aldousStoppingTime + \delta} \le X^{\varepsilon_n}_{\aldousStoppingTime} - \varepsilon_0 , \abs{X^{\varepsilon_n}_{\aldousStoppingTime}} \le M }
&\le \sup_{-M \le x \le M} \PP[ Y^x_\delta \le x - \varepsilon_0 ]\,.
\end{align}
By construction $Y^\xi$ depends on $n$ and $\tau$ (through $\xi$), while the right-hand side of \eqref{eq:displ of X bounded by Yx} does not. Thus one only needs to bound the probability of an $(\varepsilon_0 - 2\hat\varepsilon)$-displacement of  diffusions $Y^x$ with starting points $x-2\hat\varepsilon$  from a compact set, which are reflected (by ($-\hat\varepsilon$)-jumps) at constant boundaries $x-\hat\varepsilon$.
By the arguments in the proof of \cref{lemma: jump frequency} (here applied for $Y^x$ which is reflected at a constant boundary), for $\delta=\delta_0 > 0$ there exists $N\in \NN$ with the following property: for every $x\in [-M,M]$, the number $J^x_\delta:= \inf \{ k \mid k \hat\varepsilon \ge K^x_\delta \}$  of jumps  of $Y^x$  until time $\delta$ is bounded by $N-1$ with probability at least $1-\eta/8$.

Indeed, by \eqref{ineq: jump frequency -- fixed n}, fixing $\delta > 0$, $\lambda:= 1/\delta$ one gets for any $x$ 
that $\PP[J^x_\delta \ge  \ceil{N(x)}] \le \eta / 8$ where $N(x) := 1 + \paren[\big]{\log(\eta/8) - 1} / \paren[\big]{ \log \Phi_{\lambda,-}(x-\hat\varepsilon) - \log \Phi_{\lambda,-}(x) }\in \RR$.
Compactness of $[-M,M]$ and continuity of $N(x)$ gives $N:= \ceil{ \sup_{x \in [-M,M]} N(x) } < \infty$.
Hence, 
\begin{equation} \label{ineq: Y escape via N downcrossings}
	\sup_{x \in [-M,M]} \PP[ Y^x_\delta \le x - \varepsilon_0, J^x_\delta \le N-1 ] \le N \sup_{x \in [-M,M]} \PP[ \hittingTimeFromTo{x - 2\hat\varepsilon}{x - \varepsilon_0} \le \delta ],
\end{equation}
since for the event under consideration, the process $Y^x$ would have to move at least once (in at most $N$ occasions)
continuously from $x - 2\hat\varepsilon$ to $x - \varepsilon_0$.
Let $d:= (\varepsilon_0 - 2\hat\varepsilon)/2 \ge \varepsilon_0/4 > 0$, $K:= \floor{2M / d}$ and $y_k:= kd - M$.
For $x \in [y_k, y_{k+1}]$, we have $\hittingTimeFromTo{y_{k-2}}{y_{k-2}-d} \le \hittingTimeFromTo{x-\varepsilon_0}{x-2\hat\varepsilon}$ since $[y_{k-2}-d,y_{k-2}] \subset [x-\varepsilon_0, x-2\hat\varepsilon]$, so by $[-M,M] \subset [y_0,y_{K+1}]$ we get
\begin{align}
\notag
	\MoveEqLeft \PP\brackets[\big]{ \hittingTimeFromTo{X^{\varepsilon_n}_\aldousStoppingTime - \varepsilon_n}{X^{\varepsilon_n} - \varepsilon_0} \le \delta, \abs{X^{\varepsilon_n}_\aldousStoppingTime} \le M }
		\le \eta/8 + N \sup_{x \in [-M,M]} \PP[ \hittingTimeFromTo{x - 2\hat\varepsilon}{x - \varepsilon_0} \le \delta ]
\\ \notag
		&= \eta/8 + N\max_{k=0,\dots,K} \sup_{x \in [kd-M, (k+1)d-M]} \PP\brackets[\big]{ \hittingTimeFromTo{x - 2\hat \varepsilon}{x - \varepsilon_0} \le \delta }
\\ \label{ineq: X escape down probability}
		&\le \eta/8 + N\max_{k=-2,\dots,K} \PP\brackets[\big]{ \hittingTimeFromTo{y_k}{y_k - d} \le \delta } \,.
\end{align}
For a sufficiently small $\delta = \delta_1 \in (0,\delta_0]$ the right-hand side of~\eqref{ineq: X escape down probability} can be made smaller than $\eta / 4$.
The above holds for all $n$ such that $\varepsilon_n\le \hat \varepsilon$, meaning that there is some $n_0$ such that is holds for all $n\ge n_0$.
Note that $\delta_1$ only depends on $T$ (via $M$ and $K$) and on $n_0$ but not on $n$.
Hence, for all $\delta \in (0, \delta_1]$, all $n \geq n_0$ and all $\aldousStoppingTime \le T$ we have
\begin{equation} \label{ineq: X escape below}
	\PP\brackets{ X^{\varepsilon_n}_{\aldousStoppingTime+\delta} \le X^{\varepsilon_n}_\aldousStoppingTime - \varepsilon_0 } \le \frac{\eta}{2}\,.
\end{equation}

2) For the alternative second case $X^{\varepsilon_n}_{\aldousStoppingTime+\delta} \ge X^{\varepsilon_n}_\aldousStoppingTime + \varepsilon_0$,
consider the solution $(Y_t)_{t \ge \aldousStoppingTime}$ on $\rightOpenStochasticInterval{\aldousStoppingTime, \infty}$ of $\diff Y_t = b(Y_t) \diff t + \sigma(Y_t) \diff W_t$ with $Y_\aldousStoppingTime = X^{\varepsilon_n}_\aldousStoppingTime$.
Using comparison results for continuous diffusions \cite[Theorem~5.2.18]{KaratzasShreveBook} inductively over times $\rightOpenStochasticInterval{\tau^{\varepsilon_n}_{(k-1)\varepsilon_n}, \tau^{\varepsilon_n}_{k\varepsilon_n}}$, we find $Y_t \ge X^{\varepsilon_n}_t$ for all $t \in \rightOpenStochasticInterval{\aldousStoppingTime, \infty}$, a.s.
Hence, arguing like in the previous case
\begin{align}
	\notag
	\PP\brackets[\big]{ X^{\varepsilon_n}_{\aldousStoppingTime+\delta} \ge X^{\varepsilon_n}_\aldousStoppingTime + \varepsilon_0, \abs{X^{\varepsilon_n}_{\aldousStoppingTime}} \le M }
	&\le \PP\brackets[\big]{ Y_{\aldousStoppingTime+\delta} \ge Y_\aldousStoppingTime + \varepsilon_0, \abs{Y_{\aldousStoppingTime}} \le M }
\\ 
 \label{ineq: X escape up probability}
	&\le \sup_{-M \le y \le M}\PP\brackets[\big]{ \hittingTimeFromTo{y}{y + \varepsilon_0} \le \delta }
	.
\end{align}
As in \eqref{ineq: X escape down probability} we find a $\delta_2 > 0$ such that for all $\delta \in (0,\delta_2]$ the right  side of \eqref{ineq: X escape up probability} is bounded by $\eta/4$.
Hence we have $\PP[ X^{\varepsilon_n}_{\aldousStoppingTime+\delta} \ge X^{\varepsilon_n}_\aldousStoppingTime + \varepsilon_0 ] \le \eta/2$, so with \eqref{ineq: X escape below}, \cref{prop: Aldous tightness criterion} applies.
\end{proof}

\noindent
Now, to prove joint tightness of $(X^{\varepsilon_n}, L^{\varepsilon_n})_n$, we can utilize the fact that both processes satisfy Aldous' criterion and that their jump times and jump magnitudes are identical.

\begin{lemma}[Tightness of joint approximations]\label{lem: pair (Xe Le) tight}
	The sequence $(X^{\varepsilon_n}, L^{\varepsilon_n})_n$ of càdlàg $\RR^2$-valued processes defined by \labelcref{def: approx L,def: approx X} is tight.
\end{lemma}
\begin{proof}
	In view of \cref{prop: Aldous tightness criterion}, choose the space $E:= \RR^2$ equipped with Euclidean norm $\abs{\cdot}$ and let $Y^n:= (X^{\varepsilon_n}, L^{\varepsilon_n}) \in D\paren[\big]{ [0,\infty), E }$.
	Then $Y^n_0 \!=\! (-\varepsilon_n, \varepsilon_n)$ and $J_T(Y^n) \!=\! \sqrt{2}\varepsilon_n$ form tight sequences in $E$ and $\RR$, respectively.
	Furthermore,
	\[
		\PP\brackets[\big]{ \abs{Y^n_{\aldousStoppingTime+\delta} - Y^n_\aldousStoppingTime} \ge \varepsilon_0 } 
			\le \PP\brackets[\Big]{ \abs{ X^{\varepsilon_n}_{\aldousStoppingTime + \delta} - X^{\varepsilon_n}_\aldousStoppingTime } \ge \frac{\varepsilon_0}{2} } + \PP\brackets[\Big]{ \abs{L^{\varepsilon_n}_{\aldousStoppingTime+\delta} - L^{\varepsilon_n}_\aldousStoppingTime} \ge \frac{\varepsilon_0}{2} }\,.
	\]
	Hence $Y^n$ also satisfies Aldous's criterion and therefore is tight.
\end{proof}

\noindent
Tightness only implies weak convergence of a subsequence. 
It remains to show (in \cref{lemma: approx weak convergence}) that every limit point satisfies \labelcref{eq: diffusion,eq: local time} and that uniqueness in law holds. The latter will follow from pathwise uniqueness results for SDEs with reflection, while for  the former we apply results from \cite{KurtzProtter96} on weak converges of SDEs.
For that purpose, note that the approximated local times form a \emph{good} sequence of semimartingales (cf.~\cite[Definition~7.3]{KurtzProtter96}), as shown in the following
\begin{lemma} \label{lemma: approx integrator is good}
	The  sequence $(L^{\varepsilon_n})_{n}$ is of uniformly controlled variation and thus \emph{good}.
\end{lemma}
\begin{proof}
	Let $\delta:= \sup_n \varepsilon_n$. 
	Then all processes $L^{\varepsilon_n}$ have jumps of size at most $\delta < \infty$.
	Fix some $\alpha > 0$.
	By tightness, there exists some $C \in \RR$ such that $\PP\brackets{ L^{\varepsilon_n}_\alpha > C } \le 1/\alpha$.
	So the stopping time $\tau_{n,\alpha}:= \inf\braces{ t \ge 0 \mid L^{\varepsilon_n}_t > C }$ satisfies
	\(
		\PP\brackets{ \tau_{n,\alpha} \le \alpha } = \PP\brackets{ L^{\varepsilon_n}_\alpha > C } \le 1/\alpha\,.
	\)
	Moreover, by monotonicity of $L^{\varepsilon_n}$ we have 
	\(
		\EE\brackets[\Big]{ \int_0^{t \wedge \tau_{n,\alpha}} \diff\abs{L^{\varepsilon_n}}_s } 
			= \EE\brackets{ L^{\varepsilon_n}_{t \wedge \tau_{n,\alpha}} } \le C < \infty \,.
	\)
	Hence $(L^{\varepsilon_n})$ is of uniformly controlled variation in the sense of \cite[Definition~7.5]{KurtzProtter96}.
	So by \cite[Theorem~7.10]{KurtzProtter96} it is a \emph{good} sequence of semimartingales.
\end{proof}

We have gathered all necessary results to prove convergence of our approximating diffusion and local time to the continuous counterpart.
\begin{lemma}[Weak convergence of the approximations] \label{lemma: approx weak convergence}
	The sequence $\paren{ X^{\varepsilon_n}, L^{\varepsilon_n} }_{n}$ of càdlàg processes defined by \eqref{def: approx X} -- \eqref{def: approx L} converges weakly to the unique  continuous strong solution $(X,L)$ of \eqref{eq: diffusion} -- \eqref{eq: local time}.
\end{lemma}
\begin{proof}
	By Prokhorov's theorem, tightness of $(X^{\varepsilon_n}, L^{\varepsilon_n}, W)_n$ implies weak convergence of a subsequence to some limit point, $\paren{X^{\varepsilon_{n_k}}, L^{\varepsilon_{n_k}}, W}_k \weakConvergenceTo \paren{ \tilde X, \tilde L, \tilde W} \in D\paren[\big]{[0,\infty),\RR^3}$.
	Continuity of $(\tilde X, \tilde L)$ is clear since $\varepsilon_n \to 0$ is the maximum jump size.
	First we prove that $(\tilde X, \tilde L)$ satisfies the asserted SDEs.
	Afterwards, we will prove uniqueness of the limit point.
	To ease notation, let w.l.o.g.\ the subsequence $(n_k)$ be $(n)$.
	
	By \cite[Theorem~8.1]{KurtzProtter96} we get that $(\tilde X, \tilde L)$ satisfy~\eqref{eq: diffusion} for the semimartingale $\tilde W$. 
	That $\tilde W$ is a Brownian motion follows from standard arguments, cf.~\cite[proof of Theorem~1.9]{NystromOnskog10}. 
	As $D\paren[\big]{[0,\infty),\RR^3}$ is separable we find, by an application of the Skorokhod representation theorem, that $\tilde L$ is non-decreasing and $\tilde X_t \le g(\tilde L_t)$ for all $t \ge 0$, $\PP$-a.s. because these properties already hold for $(X^{\varepsilon_n}, L^{\varepsilon_n})$.
	
	To prove that $\tilde L$ grows only at times $t$ with $\tilde X_t = g(\tilde L_t)$, we have to approximate the indicator function by continuous functions.
	For $\delta > 0$ define
	\begin{align*}
		h_\delta(x,\ell) &:= \begin{cases}
					\paren[\big]{x-g(\ell)}/\delta + 1 &\text{for } g(\ell) - \delta \le x \le g(\ell),
				\\	1 - \paren[\big]{x-g(\ell)}/\delta &\text{for } g(\ell) \le x \le g(\ell) + \delta,
				\\	0 &\text{otherwise},
				\end{cases}
\end{align*}
\[
	h_0(x,\ell) := \indicator_{\braces{ x=g(\ell) }} 
	\;\text{and} \;
	H^{\delta,n}_t := h_\delta( X^{\varepsilon_n}_t, L^{\varepsilon_n}_t )
	\;\text{and} \;
	\tilde H^\delta_t := h_\delta( \tilde X_t , \tilde L_t )\,.
\]	

For $\delta \searrow 0$ the functions $h_\delta \searrow h_0$ converge pointwise  monotonically.
Continuity of $h_\delta$ implies weak convergence $\paren{ H^{\delta,n}, L^{\varepsilon_n} } \weakConvergenceTo \paren{ \tilde H^\delta, \tilde L }$.
By \cref{lemma: approx integrator is good}, $(L^{\varepsilon_n})$ is a good sequence. So for every $\delta > 0$ the stochastic integrals
	\(
		\int_0^\cdot H^{\delta,n}_{s-} \diff L^{\varepsilon_n}_s \weakConvergenceTo \int_0^\cdot \tilde H^\delta_{s-} \diff \tilde L_s\,
	\)
converge weakly.
	Note that $\diff L^{\varepsilon_n}_t = H^{0,n}_{t-} \diff L^{\varepsilon_n}_t$.
	Hence, for every $\delta > 0$ we have
	\[
		\int_0^\cdot H^{\delta,n}_{s-} \diff L^{\varepsilon_n}_s = \int_0^\cdot H^{\delta,n}_{s-} H^{0,n}_{s-} \diff L^{\varepsilon_n}_s = \int_0^\cdot H^{0,n}_{s-} \diff L^{\varepsilon_n}_s = L^{\varepsilon_n}\,.
	\]
	With weak convergence $L^{\varepsilon_n} \weakConvergenceTo \tilde L$ it follows for every $\delta > 0$ that
	\(
		\tilde L_t = \int_0^t \tilde H^\delta_{s-} \diff \tilde L_s\,.
	\)
	By monotonicity of $\tilde L$, $\diff \tilde L_t$ defines a random measure on $[0,\infty)$.
	Hence monotone convergence of $\tilde H^\delta_t \searrow \tilde H^0_t$ yields
	\(
		\diff \tilde L_t = h_0(\tilde X_t, \tilde L_t) \diff \tilde L_t\,.
	\)

	Thus, we showed that $(X^\varepsilon, L^\varepsilon)$ converges in distribution to a weak solution $(\tilde X, \tilde L)$ of the reflected SDE, i.e.~it might be defined on a different probability space with its own Brownian motion. Note that $(\tilde X, \tilde L)$ is continuous on $[0,\infty)$ and that $\tilde \tau_\infty := \sup_{k} \tilde \tau_k = \infty$ a.s., where $\tilde \tau_k := \inf \braces{ t > 0 \mid \abs{\tilde X_t} \vee \tilde L_t > k }$. To show the existence and  uniqueness of a strong solution as stated in the theorem, we will use the results from \cite{DupuisIshii93}.
	Consider the domain~$\bar G := \braces{ (x,\ell) \in \RR^2 \mid x \le g(\ell), \ell \ge 0 }$.
	We may interpret the process $(X_t, L_t)$ as a continuous diffusion in $\bar G$ with oblique reflection in direction $(-1,+1)$ at the boundary, although the notion of a two-dimensional reflection seems unusual here, because $(X,L)$ only varies in one dimension in the interior of $G$.
	The unbounded domain $G$ can be exhausted by bounded domains $G_k:= \braces[\big]{ (x,\ell) \in G \bigm| \abs{x},\abs{\ell} < k }$, which might have a non-smooth boundary especially at $\paren{ g(0), 0 }$, but still satisfy \cite[Cond.~(3.2)]{DupuisIshii93}. 
	Hence, by \cite[Cor.~5.2]{DupuisIshii93} the process $(X,L)$  exists (up to explosion time) on the initial probability space and is (strongly) unique on $\rightOpenStochasticInterval{0, \tau_k}$ with exit time $\tau_k:= \inf \braces{ t > 0 \mid \abs{X_t} \vee L_t > k }$, for all $k\in \NN$. 
	So $(X,L)$ is unique until explosion time $\tau_\infty:= \sup_k \tau_k$. 
	Moreover, by \cite[Theorem~5.1]{DupuisIshii93} we have the following pathwise uniqueness result: for any two continuous solutions $(X^1, L^1)$ and $(X^2, L^2)$ with explosion times $\tau^1_\infty$ and $\tau^2_\infty$, respectively defined on the same probability space with the same Brownian motion and the same initial condition, we have that $X^1 = X^2 $ and $L^1 = L^2$ on $\closedStochasticInterval{0,\tau^1_k\wedge \tau^2_k}$ for every $k\in \NN$ a.s. 
	Using a known argument due to Yamada and Watanabe, ideas being as in \cite[Ch.~5.3.D]{KaratzasShreveBook}, one can bring the two (weak) solutions $(\tilde X, \tilde L, \tilde W)$ and $(X, L, W)$ to a canonical space with a common Brownian motion. 
	By pathwise uniqueness there, one concludes that $\tau_\infty = \infty$ a.s.~(as $\tilde\tau_\infty = \infty$). 
	Hence the strong solution $(X,L)$ does not explode in finite time. 
	In addition, we conclude uniqueness in law like in \cite[Prop.~5.3.20]{KaratzasShreveBook} and thus any weak limit point of the approximating sequence $(X^\varepsilon, L^\varepsilon)$ will have the same law as $(X, L)$.
\end{proof}

This convergence result can be strengthened as follows.
\begin{corollary}[Convergence in probability] \label{lemma: approx ucp}
	The sequence $\paren{ X^{\varepsilon_n}, L^{\varepsilon_n} }_{n}$ of càdlàg processes defined by \eqref{def: approx X}--\eqref{def: approx L} converges in probability to $(X,L)$ defined by \eqref{eq: diffusion}--\eqref{eq: local time}.
\end{corollary}
\begin{proof}
	Following the proof of \cite[Cor.~5.6]{KurtzProtter91}, we will strengthen weak convergence $(X^{\varepsilon_n}, L^{\varepsilon_n}) \weakConvergenceTo (X,L)$ to convergence in probability.
	First, note \cref{lemma: approx weak convergence} implies weak convergence of the triple $(X^{\varepsilon_n}, L^{\varepsilon_n}, W) \weakConvergenceTo (X,L,W)$ by e.g.~\cite[Corollary~3.1]{SchiopuKratina85}.
	Hence, for every bounded continuous $F: D([0,\infty);\RR^2) \to \RR$ and every bounded continuous $G: C([0,\infty);\RR) \to \RR$, we have
	$
		\lim_{n \to \infty} \EE\brackets{ F(X^{\varepsilon_n}, L^{\varepsilon_n}) G(W) } = \EE[F(X,L)G(W) ]
		\,.
	$
	Now, the previous equation even holds for all bounded measurable $G$ by $L^1$-approximation of measurable functions by continuous functions.
	By strong uniqueness of $(X,L)$, there exists a measurable function $H: C([0,\infty);\RR) \to C([0,\infty);\RR^2)$ such that $(X,L) = H(W)$.
	In particular, $G(W):= F(H(W)) = F(X,L)$ is bounded and measurable, so we conclude
	\begin{align*}
		\MoveEqLeft
		\lim_{n \to \infty} \EE\brackets[\big]{\paren{ F(X^{\varepsilon_n}, L^{\varepsilon_n}) -  F(X,L)}^2}
	\\
			&= \lim_{n \to \infty} \paren[\big]{ \EE\brackets[\big]{ F(X^{\varepsilon_n}, L^{\varepsilon_n})^2 } - 2 \EE\brackets[\big]{ F(X^{\varepsilon_n}, L^{\varepsilon_n}) F(X,L) } + \EE\brackets[\big]{ F(X,L)^2 } }
			=0
	\end{align*}		
	and hence convergence in probability follows.
\end{proof}

\begin{corollary}[Weak convergence of the inverse local times] \label{lemma: tau weak convergence}
For any $\ell > 0$, the sequence $\paren{ \tau^{\varepsilon_n}_\ell }_n$ from~\eqref{def: approx tau} converges in law to the inverse local time $\tau_\ell$ defined by~\eqref{def: tau}.
\end{corollary}
\begin{proof}
	Convergence $L^{\varepsilon_n} \weakConvergenceTo L$ implies $L^{\varepsilon_n}_t \weakConvergenceTo L_t$ at all continuity points of $L$, i.e.~at all points, hence
	\(
		\PP\brackets[\big]{ \tau^{\varepsilon_n}_\ell \le t } 
		= \PP\brackets[\big]{ L^{\varepsilon_n}_t \ge \ell } 
		\to
		\PP\brackets{ L_t \ge \ell } 
		= \PP\brackets{ \tau_\ell \le t }\,.
	\)
\end{proof}

\noindent
This completes the proof of \cref{thm: local time Laplace}.

\end{document}